\documentclass{amsart}
\usepackage{amsmath, amsfonts, amssymb,amsthm}

\newtheorem{lemma}{Lemma}

\newtheorem*{theorem*}{Theorem}
\newtheorem*{main}{Main Theorem}
\newtheorem*{first}{First Main Theorem}
\newtheorem*{theorem R}{Theorem R}

\theoremstyle{definition}

\newcommand{\C}{\mathbb{C}}

\newcommand{\N}{\mathbb{N}}

\newcommand{\PP}{\mathbb{P}}
\newcommand{\ii}{{\bf{i}}}

\newcommand{\Z}{\mathbb{Z}}

\let\a\alpha   
  \let\g\gamma

\begin{document}
\title[An explicit estimate on multiplicity truncation]{An explicit estimate on multiplicity truncation in 
the second main theorem for holomorphic curves encountering hypersurfaces
in general position in projective space}
\author{Ta Thi Hoai An$^*$}
\address{Institute of Mathematics\\
18 Hoang Quoc Viet, Cau Giay\\
Hanoi, VietNam}
\email{tthan@math.ac.vn}
\author{Ha Tran Phuong}
\address{Department of Mathematics, Thai Nguyen University of education}
\email{hatranphuong@yahoo.com}
\thanks{*Junior Associate Member of ICTP, Trieste, Italy}
 \subjclass[2000]{32H30}
\begin{abstract}
Yan and Chen proved a weak Cartan-type second main theorem for 
holomorphic curves meeting hypersurfaces in projective
space that included truncated counting functions.
Here we give an explicit estimate for the level of truncation.
\end{abstract}
\maketitle 

\section{ Introduction}
We begin by recalling some standard notation 
and results from Nevanlinna theory.

Let  $f: \C \rightarrow \PP^n(\C)$  be a holomorphic map. Let $f=(f_0: \dots :f_n)$ be a reduced representative of $f$, where $f_0,\dots, f_n$ are  entire functions on $\C$ without common zeros. The Nevanlinna-Cartan characteristic function $T_f(r)$ is defined by
$$T_f(r) = \dfrac{1}{2\pi} \int_0^{2\pi} \log\|f(re^{i\theta})\|d\theta,$$
where $\|f(z)\| = \max\{|f_0(z)|,\dots ,|f_n(z)|\}$.
The above definition is independent, up to an additive constant, of the choice of the reduced representation of $f$.
Let $D$ be a hypersurface in $\PP^n(\C)$ of degree $d$. Let $Q$ be the homogeneous polynomial of degree $d$ defining $D$. The proximity function of $f$ is defined by
$$ m_f(r,D) = \dfrac{1}{2\pi}\int_0^{2\pi} \log\dfrac{\| f(re^{i\theta})\|^d}{|Q\circ f(re^{i\theta})|} d\theta.$$
The above definition is independent, up to an additive constant, of the choice of the reduced representation of $f$ and the choice of the defining
polynomial $Q$.
Let $n_f(r,D)$ be the number of zeros of $Q\circ f$ in the disk $|z|\le r$, counting multiplicity, and $n_f^M(r,D)$ be the number of zeros of $Q\circ f$ in the disk $|z| \le r$, where any zero of multiplicity greater than $M$
is ``truncated'' and counted as if it only had multiplicity $M.$
The integrated counting and truncated counting functions are defined by
$$N_f(r,D) = \int_0^r \dfrac{n_f(t,D)- n_f(0,D)}{t} dt - n_f(0,D)\log r.$$
$$N_f^M(r,D) = \int_0^r \dfrac{n_f^M(t,D)- n_f^M(0,D)}{t} dt - n_f^M(0,D)\log r.$$
When we want to emphasize $Q,$ we sometimes also write 
$N_f(r,D)$ as $N_f(r,Q)$ and $N^M_f(r,D)$ as  $N^M_f(r,Q).$

A consequence of the Poisson-Jensen formula is the following:
 
\begin{first}  Let $f: \C \rightarrow \PP^n(\C)$ be a holomorphic map, and D be a hypersurface in $\PP^n(\C)$ of degree d. If $f(\C) \not\subset D$, then for every real number r with $0 < r < \infty$
$$m_f(r,D) + N_f(r,D) = dT_f(r) + O(1),$$
where $O(1)$ is a constant independent of $r$.
\end{first}

Recall that hypersurfaces $D_1, \dots,D_q, q > n$, in $\PP^n(\C)$
are said to be in general position if for any distinct $i_1,\dots , i_{n+1} \in \{1,\dots,q\}$,
$$\bigcap_{k=1}^{n+1}\mathrm{supp}(D_{i_k}) = \emptyset.$$

Yan and Chen \cite{YC} proved:

\begin{theorem*}[Yan and Chen] Let $f: \C \rightarrow \PP^n(\C)$ be an algebraically non-degenerate holomorphic map, and let $D_j, 1 \leqslant j \leqslant q,$ be hypersurfaces in $\PP^n(\C)$ of degree $d_j$ in general position. Then for every $\varepsilon >0$, there exists a positive integer $M$ such that
$$ (q-(n+1)-\varepsilon )T_f(r) \leqslant \sum_{j=1}^q d_j^{-1} N_f^M(r,D_j) + o(T_f(r)),$$
where inequality holds for all large $r$ outside a set of finite Lebesgue measure.
\end{theorem*}

Note that this is how Yan and Chen state their theorem, but clearly
the $o(T_f(r))$ term can be incorporated into the $\varepsilon T_f(r)$
term on the left.

Yan and Chen's work was based on Ru's \cite{R1} solution of Shiffman's
conjecture, which in turn was based on a new technique of
Corvaja and Zannier \cite{CZ1}.  Yan and Chen's contribution was to
incorporate multiplicity truncation into the counting functions.
When one applies inequalities of second main theorem type, it is often
crucial to the application to have the inequality with truncated counting
functions.  For example, all existing constructions of unique range sets
depend on a second main theorem with truncated counting functions.
Obtaining a number theoretic analog of the second main theorem that includes
truncated counting functions is perhaps one of the most important open problems
in Diophantine approximation.

In a recent preprint, Dethloff and Tran \cite{DT} also include
truncated counting functions in a weak Cartan-type second main theorem
inequality for hypersurfaces in projective space, and they also treat
the case of slowly moving hypersurfaces.

Our purpose in this paper is to give an explicit estimate on the 
truncation level $M$ in Yan and Chen's theorem.
For concrete applications, it is often important
to know explicity at what level the multiplicity in the counting functions
can be truncated.  Our main theorem is:

\begin{main} Let $f: \C \rightarrow \PP^n(\C)$ be an algebraically non-degenerate holomorphic map, and let $D_j, 1 \leqslant j \leqslant q,$ be hypersurfaces in $\PP^n(\C)$ of degree $d_j$ in general position.
Let $d$ be the least common multiple of the $d_j$'s. 
Let $0<\varepsilon <1$ and let
$$M\ge 2d\lceil2^n(n+1)n(d+1)\varepsilon^{-1}\rceil^n$$
Then,
$$ (q-(n+1)-\varepsilon )T_r(f) \leqslant \sum_{j=1}^q d_j^{-1} N_f^{M}(r,D_j),$$
where inequality holds for all large $r$ outside a set of finite Lebesgue measure. 
\end{main}

The main tool which allows us to estimate the truncation level is
the more detailed analysis of Corvaja and Zannier's filtration that
was done by the first author and Wang in \cite{AW}.

Unfortunately, our estimate for $M$ depends on $\varepsilon,$ but for typical
applications one can generally fix a small $\varepsilon$ once and for all,
so our estimate is still useful.

Either by bringing in the standard machinery of value distribution theory
for several variables or by a standard argument averaging over complex
lines, we have the same theorem for meromorphic mappings from
$\C^m$ to $\PP^n(\C).$  For the sake of simplicity, we prefer to work
only with maps from $\C$ leaving the routine details of extending the result
to maps from $\C^m$ to the reader.

We also remark that our main theorem is also true for 
non-Archimedean analytic curves from a complete non-Archimedean
field of characteristic zero.  The algebraic argument involving
the Corvaja and Zannier filtration is identical, and it is simply
a matter of using the non-Archimedean hyperplane second main theorem
in place of Ru's result that we refer to as Theorem~R below.
See \cite{Boutabaa}, \cite{KT}, \cite{CY}, or \cite{HY} for
the non-Archimedean hyperplane second main theorem.  Although each of
those references state the theorem for hyperplanes in general position,
reformulating the result as in Theorem~R below is a routine exercise
left for the reader.

\section{Corvaja and Zannier's filtration and Ru's Hyperplane Theorem} 
\def\theequation{2.\arabic{equation}}
\setcounter{equation}{0} 

In this section we recall Corvaja and Zannier's filtration as made more
explicit in \cite{AW}, and we recall Ru's theorem of second main theorem
type for hyperplanes, which we will need for our proof.
Details of proofs can be found in 
 \cite{CZ1}, \cite{R1}, and also \cite{AW}.

For a fixed big integer $\a$, denote by $V_\a$ the space of homogeneous polynomials of degree $\a$ in $\C[x_0,\dots,x_n]$. 

\begin{lemma} [see \cite{AW}]  \label{lm1} Let $\gamma_1,\dots,\gamma_n$ be homogeneous polynomials in $\C[x_0,\dots,x_n]$ and assume that they define a subvariety of $\PP^n(\C)$ of dimension 0. Then for all $\a \geqslant \sum\limits^n_{i=1} \deg \gamma_i$,
$$ \dim \dfrac{V_\a}{(\gamma_1,\dots,\gamma_n) \cap V_{\a}} = \deg \gamma_1\dots \deg \gamma_n.$$ 
\end{lemma}

Throughout of this paper, we shall use the {\it lexicographic ordering}
on $n$-tuples $(i_1,\dots,i_n) \in \N^n$ of natural numbers. Namely, $(j_1,\dots,j_n)> (i_1,\dots,i_n) $ if and only if for some $b \in \{1,\dots,n\}$ we have $j_l = i_l$ for $l<b$ and $j_b > i_b$. 
Given an $n$-tuples $ ({\ii}) = (i_1,\dots,i_n)$ of non-negative integers, we denote $\sigma({\ii}):= \sum_j i_j.$

Let $\gamma_1,\dots,\gamma_n \in \C[x_0,\dots,x_n]$ be the homogeneous polynomials of degree $d$ that define a zero-dimensional subvariety of $\PP^n(\C).$
We now recall Corvaja and Zannier's filtration of $V_\a.$
 Arrange, by the lexicographic order, the $n$-tuples $ ({\ii}) = (i_1,\dots,i_n)$ of non-negative integers such that $\sigma({\ii}) \leqslant \a/d.$ Define the spaces $W_{(\ii)} = W_{N,{(\ii)}}$ by
$$W_{(i)} = \sum_{({\bf e})\geqslant ({\bf i})} \gamma_1^{e_1}\dots\gamma_n^{e_n}V_{\a - d\sigma({\bf e})}.$$
Clearly,  $W_{(0,\dots,0)} = V_\a$ and $W_{({\bf i})} \supset W_{({\bf i'})}$ if $({\bf i'}) > ({\bf i})$, so the $W_{({\bf i})}$ is a filtration of $V_\a$. 

Next, we recall a result about the quotients of consecutive spaces
in the filtration.

\begin{lemma} [see \cite{AW}]  \label{lm2}There is an isomorphism $$\dfrac{W_{({\bf i})}}{W_{({\bf i'})}} \cong \dfrac{V_{\a-d\sigma({\bf i})}}{(\gamma_1,\dots,\gamma_n) \cap V_{\a-d\sigma({\bf i}) }}.$$
Furthermore, we may choose a basis of $\dfrac{W_{({\ii})}}{W_{({\ii'})}}$ from the set containing all equivalence classes of the form: $\gamma_1^{i_1}\dots\gamma_n^{i_n}\eta$ modulo  $W_{({\ii'})}$ with $\eta$ being a monomial in $x_0,\dots,x_n$ with total degree $\a-d\sigma({\ii})$.
\end{lemma}

We now combine Lemma~\ref{lm1} and Lemma~\ref{lm2} to explicitly calculate the dimension
of the quotient spaces, which we denote by $\Delta_{({\ii})}$.

\begin{lemma} \label{lm3}If  $\sigma({\ii}) \leqslant \a/d - n$ then
$$\Delta_{(\ii)}:= \dim \dfrac{W_{({\bf i})}}{W_{({\bf i'})}} = d^n.$$
\end{lemma}

\begin{proof} Since the polynomials $\gamma_1,\dots,\gamma_n$ have 
the same degree $d$, it follows 
$$\dim \dfrac{W_{({\bf i})}}{W_{({\bf i'})}} = \dim \dfrac{V_{\a-d\sigma({\bf i})}}{(\gamma_1,\dots,\gamma_n) \cap V_{\a-d\sigma({\bf i}) }} = d^n$$
if $\a - d\sigma({\ii}) \geqslant nd$.
\end{proof}

To prove our result, we also need the following general form of the main theorem for holomorphic curves intersecting hyperplanes which was given by Ru in \cite{R2}.

\begin{theorem R}[Ru \cite{R2}]  Let $f=(f_0:\dots:f_n): \C \rightarrow \PP^n(\C)$ be a holomorphic map whose image it not contained in any proper linear subspace. Let $H_1,\dots,H_q$ be arbitrary hyperplanes in $\PP^n(\C)$. Let $L_j, 1 \leqslant j \leqslant q,$ be the linear forms defining $H_1,\dots,H_q$. Denote by $W(f_0,\dots,f_n)$ the Wronskian of $f_0,\dots,f_n$. Then,
$$ \int_0^{2\pi} \max_K \log \prod_{j \in K} \dfrac{\|f(re^{i\theta})\| \|L_j\|}{| L_j(f)(re^i\theta)|}\dfrac{d\theta}{2\pi} + N_W(r,0) \leqslant (n+1)T_f(r) +o(T_f(r)), $$
where the maximum is taken over all subsets $K$ of $\{1,\dots,q\}$ such that the linear forms $L_j, j \in K$, are linearly independent, and $\|L_j\|$ is the maximum of the absolute values of the coefficients in $L_j$.
\end{theorem R}

\section{Proof of the Main Theorem.}
\def\theequation{3.\arabic{equation}}
\setcounter{equation}{0} 

Let $f=(f_0:\dots:f_n): \C \rightarrow \PP^n(\C)$ be an algebraically non-degenerate holomorphic map, 
and let $D_j, 1 \leqslant j \leqslant q,$ be hypersurfaces in $\PP^n(\C)$ of degree $d_j$ in general position. 
Let $Q_j, 1 \leqslant j \leqslant q,$ be the homogeneous polynomials in $\C[x_0,\dots,x_n]$ of degree $d_j$ defining $D_j$. Of course we may assume
that $q\ge n+1.$

We first claim that it suffices to prove the theorem in the case
that all of the $d_j$ are equal to $d.$  Indeed, if we have the theorem
in that case, then we know that for $\varepsilon$ and $M$ as in the
statement of the theorem that
$$ (q-(n+1)-\varepsilon )T_r(f) \leqslant \sum_{j=1}^q d^{-1} N_f^{M}(r,Q_j^{d/d_j}),$$
where $d$ is the least common multiple of the $d_j.$
Note that if $z \in \C$ is a zero of $Q_j \circ f$ with multiplicity $a,$
then $z$ is zero of $Q_j^{d/d_j} \circ f$ with multiplicity $a\dfrac{d}{d_j}$. This implies that
$$N_f^M(r,Q_j^{\frac{d}{d_j}})
 \leqslant \dfrac{d}{d_j} N_f^{M}(r,Q_j).$$
Hence,
$$ (q-(n+1)-\varepsilon )T_r(f) \leqslant \sum_{j=1}^q d^{-1} N_f^{M}(r,Q_j^{d/d_j}) \leqslant \sum_{j=1}^q d_j^{-1} N_f^{M}(r,Q_j).$$
Therefore, without loss of generality, we may assume that $Q_1, \dots , Q_q$ have
the same degree $d$.

Given $z \in \C$, there exists a renumbering $\{i_1, \dots ,i_q\}$ of the indices $\{1, \dots ,q\}$ such that 
\begin{equation}
|Q_{i_1}\circ f(z)| \leqslant |Q_{i_2}\circ f(z)| \leqslant  \dots  \leqslant |Q_{i_q}\circ f(z)|. \label{3.1}
\end{equation}
Since $Q_{i_j}, 1 \leqslant j \leqslant q $ are in general position, by Hilbert's Nullstellensatz \cite{V}, for any integer $k, 0 \leqslant k \leqslant n $, there is an integer $m_k \geqslant d$ such that
$$x_k^{m_k} = \sum_{j=1}^{n+1} b_{k_j} (x_0,..,x_n) Q_{i_j}(x_0, \dots ,x_n),$$
where $b_{k_j}, 1 \leqslant j \leqslant n+1, 0 \leqslant k \leqslant n,$ are
homogeneous forms with coefficients in $\C$ of degree $m_k - d$. So, 
$$|f_k(z)|^{m_k} \leqslant c_1 \|f(z)\|^{m_k-d} \max \{|Q_{i_1}\circ f(z)|, \dots ,|Q_{i_{n+1}}\circ f(z)| \},$$
where $c_1$ is a positive constant depends only on the coefficients of $b_{k_j}, 1 \leqslant j \leqslant n+1, 0 \leqslant k \leqslant n,$ thus depends only on the coefficients of $Q_i, 1 \leqslant i \leqslant q.$  Therefore, 
\begin{equation}
\|f(z)\|^{d} \leqslant c_1 \max \{|Q_{i_1}\circ f(z)|, \dots ,|Q_{i_{n+1}}\circ f(z)| \}.\label{3.2}
\end{equation}
By (\ref{3.1}) and (\ref{3.2}),
$$ \prod_{j=1}^q \dfrac{\|f(z)\|^d}{|Q_{i_j}\circ f(z)|} \leqslant c_1^{q-n} \prod_{j=1}^n \dfrac{\|f(z)\|^d}{|Q_{i_j}\circ f(z)|}.$$
Hence, by the definition,
\begin{equation}
\sum_{j=1}^q m_f(r,Q_j) \leqslant \int_0^{2\pi} \max_{\{i_1, \dots ,i_n\}}  \log \prod_{j=1}^n \dfrac{\|f(re^{i\theta})\|^d}{|Q_{i_j}(f)(re^{i\theta})|} \dfrac{d\theta}{2\pi} + (q-n) \log c_1. \label{3.3}
\end{equation}

Pick $n$ distinct polynomials $\g_1, \dots ,\g_n\in \{Q_1,  \dots , Q_q\}$. By the  general position 
assumption, they define  a subvariety of dimension $0$ in $\PP^n$.  For a
fixed  integer 
$\alpha$ ($\ge  nd $), which will be chosen later, let 
$V_\alpha$ be the space of homogeneous polynomials of degree $\alpha$ in $\C[x_0, \dots ,x_n]$.   
In the previous section we recalled a 
filtration $W_{({\ii})}$ of $V_\alpha$ with
$$\Delta_{(\ii)}:= \dim \dfrac{W_{({\bf i})}}{W_{({\bf i'})}} = d^n,$$
 for any $({\ii'}) >({\ii})$ consecutive $n$-tuples
with $\sigma(\ii)\le\alpha/d-n.$

Set $M = \dim V_\alpha$.  We will see that we will be able to truncate
our counting functions to level $M,$ and we will later give an estimate
on $\dim V_\alpha.$

We now 
recall Corvaja and Zannier's choice of a
suitable basis $\{\psi_1, \dots ,\psi_M\}$ for $V_\alpha.$
Start with the last nonzero $W_{({\ii})}$ and pick any basis for it.
Then continue inductively as follows: suppose $({\ii'}) >({\ii})$ are
consecutive $n$-tuples such that $d\sigma({\ii}), d\sigma({\ii'}) \leqslant \alpha$ and assume that we have chosen a basis of $W_{({\ii'})}$. It follows directly from the definition that we may pick representatives in $W_{({\ii})}$ for the quotient space $W_{({\ii})}/W_{({\ii'})}$, of the form $\gamma_1^{i_1} \dots \gamma_n^{i_n}\eta$, where $\eta \in V_{\alpha - d\sigma({\ii})}$. Extend the previously constructed basis in $W_{({\ii'})}$ by adding these representatives. In particular, we have obtained a basis for $W_{({\ii})}$ and our inductive procedure may go on unless $W_{({\ii})} = V_\alpha$, in which case we stop. In this way, we have obtained a basis $\{\psi_1, \dots ,\psi_M\}$ for $V_\alpha$.

Let $\phi_1, \dots ,\phi_M$ be a fixed basis of $V_\alpha$. Then $\{\psi_1, \dots ,\psi_M\}$ can written as linear forms $L_1, \dots ,L_M$ in $\phi_1, \dots ,\phi_M$ so that $\psi_t(f) = L_t(F)$, where
$$
	F=(\phi_1(f): \dots :\phi_M(f)):\C\to \PP^{M-1}.
$$
The linear forms $L_1,\dots,L_M$ are linearly independent, and we know from the 
assumption that $f$ is algebraically non-degenerate
that $F$ is linearly non-degenerate.  

For $z\in \C$, we now estimate $\log \prod_{t=1}^M|L_t(F)(z)| = \log \prod_{t=1}^M|\psi_t(f)(z)| $. Let $\psi$ be an element of the basis, constructed with respect to $W_{({\ii})}/W_{({\ii'})}.$ Then, we have $\psi = \gamma_1^{i_1}\dots\gamma_n^{i_n}\eta$, where $\eta \in V_{\alpha-d\sigma({\bf i})}$. 
Thus, we have a bound
\begin{align*}
|\psi\circ f(z)| &\leqslant |\gamma_1 \circ f(z)|^{i_1}\dots|\gamma_n \circ f(z)|^{i_n}|\eta \circ f(z)|\\
& \leqslant c_2 |\gamma_1 \circ f(z)|^{i_1}\dots|\gamma_n  \circ f(z)|^{i_n}\|f(z)\|^{\alpha - d \sigma({\bf i})},
\end{align*}
where $c_2$ is a positive constant that depends only on $\psi$, 
but not on $f$ or $z$. Observe that there are precisely $\Delta_{({\bf i})}$ such functions $\psi$ in our basis. 
Hence,
\begin{align*}\log  |\psi_t \circ f(z)|& \leqslant i_1 \log|\gamma_1 \circ f(z)| +\dots+ i_n \log | \gamma_n \circ f(z)| + (\alpha - d \sigma({\bf i})) \log \|f(z)\|+c_3\\
&\le i_1\Big(\log |\gamma_1 \circ f(z)|-\log \|f(z)\|^d\Big)+\dots+i_n\Big(\log |\gamma_1 \circ f(z)|-\log \|f(z)\|^d\Big)\\
&\quad \ \ +\alpha \log \|f(z)\|+c_3\\
&\le -i_1\log \dfrac{\|f(z)\|^d}{| \gamma_1\circ f(z)|}-\dots-i_n\log \dfrac{\|f(z)\|^d}{| \gamma_n \circ f(z)|}+\alpha \log \|f(z)\|+c_3.
\end{align*}
Therefore,
\begin{align*}
\log \prod_{t=1}^M|L_t\circ F(z)| &=\log \prod_{t=1}^M |\psi_t \circ f(z)| \cr
& \leqslant - \sum_{({\bf i})} \Delta_{({\bf i})} \Big(i_1\log \dfrac{\|f(z)\|^d}{| \gamma_1 \circ f(z)|}+\cdots+i_n\log \dfrac{\|f(z)\|^d}{| \gamma_n \circ f(z)|}\Big)\cr
&\quad \ +M\alpha \log \|f(z)\|+Mc_3\cr
&\le-\Big(\sum_{({\bf i})} \Delta_{({\bf i})} i_j\Big )\sum_{j=1}^n\log\dfrac{\|f(z)\|^d}{| \gamma_j(f)(z)|}+M\alpha \log \|f(z)\|+Mc_3\cr
&\le -\Delta\log\prod_{j=1}^n \dfrac{\|f(z)\|^d}{| \gamma_j(f)(z)|}+M\alpha \log \|f(z)\|+Mc_3, 
\end{align*}
where  the summations are taken over the $n$-tuples with $\sigma{({\bf i})} \leqslant \alpha/d$; and $\Delta:=\sum_{({\bf i})} \Delta_{({\bf i})} i_j$. This implies 
\begin{align}
 \log \prod_{j=1}^n \dfrac{\|f(z)\|^d}{| \gamma_j \circ f(z)|} \leqslant &\frac 1\Delta\log \prod_{t=1}^M \dfrac{\|F(z)\|}{| L_t \circ F(z)|}-\frac M\Delta\log\|F(z)\| 
+ \frac{M\alpha}\Delta\log \|f(z)\|+ c_4. \label {3.10}
\end{align}

Since there are only finitely many choices $\{\gamma_1, \dots ,\gamma_n\} \subset \{Q_1, \dots , Q_q\}$, we have a finite collection of linear forms $L_1, \dots ,L_u$. From (\ref{3.10}) we have
\begin{align*}
& \int_0^{2\pi} \max_{\{i_1, \dots ,i_n\}} \log \prod_{k=1}^n \dfrac{\|f(re^{i\theta})\|^d}{|Q_{i_k}(f)(re^{i \theta})|}\dfrac{d\theta}{2\pi}\\
&\leqslant \frac 1\Delta \int_0^{2\pi} \max_{K} \log \prod_{j \in K} \dfrac{\|F(re^{i\theta})\| \|L_j\|}{|L_j(F)(re^{i \theta})|}\dfrac{d\theta}{2\pi}- \frac M\Delta T_F(r) + \frac{M\alpha}\Delta T_f(r) + c_5,
\end{align*}
where $\max_K$ is taken over all subsets $K$ of $\{1, \dots ,u\}$ such that linear forms $L_j, j \in K$, are linearly independent, 
and $c_5$ is constant independent of $r$. 
Applying Theorem~R to the holomorphic map $F:\C\to \PP^{M-1}$ 
and the hyperplanes defined by the linear forms $L_1, \dots ,L_u$, and then 
using (\ref{3.3}), and so we have
\begin{align*}
\sum_{i=1}^qm_f(r, Q_i)\le& \int_0^{2\pi} \max_{\{i_1, \dots ,i_n\}} \log \prod_{k=1}^n \dfrac{\|f(re^{i\theta})\|^d}{|Q_{i_k}(f)(re^{i \theta})|}\dfrac{d\theta}{2\pi}+c_6\\ 
&\leqslant -\frac1\Delta N_W(r,0) + o(T_F(r)) +\frac{M\alpha}\Delta T_f(r),
\end{align*}
where $W$ is the Wronskian of $F_1, \dots ,F_M$. 
By the First Main Theorem, 
$$
	T_F(r) \leqslant \alpha T_f(r) + O(1),
$$
and hence we have
\begin{align}
&(qd- \frac{M\alpha}\Delta ) T_f(r) \leqslant  \sum_{j=1}^q N_f(r,Q_j) -\frac 1\Delta N_W(r,0)+o(T_f(r)). \label{3.11}
\end{align}

We will now estimate the left hand side of (\ref{3.11}).
Since the number of non-negative integer $m$-tuples with sum $\leqslant S$ is equal to the number  of non-negative integer $(m+1)$-tuples with sum exactly $S \in \Z$, which is $\binom{S+m}{m}$. Assume $\alpha$ is divisible by $d$, it follows  from Lemma~3  that,
\begin{align*}\Delta:=\sum_{\sigma({\bf i})\le \alpha/d}i_j\Delta_{({\bf i})}
&\ge  \sum_{\sigma({\bf i})\le \alpha/d-n}i_j\Delta_{({\bf i})}
=\frac {d^n}{n+1}\sum_{({\bf \widehat i})}\sum_{j=1}^{n+1} i_j\cr
&=\frac{d^n}{n+1}\sum_{({\bf \widehat i})}(\alpha/d-n)
=\frac{d^n}{n+1}{\alpha/d \choose n}(\alpha/d-n)\cr
&=\frac{\alpha(\a-d)\cdots(\a-dn)}{d(n+1)!},
\end{align*}
where the sum $\sum_{({\bf
\widehat i})}$ is taken over the nonnegative integer $(n+1)$-tuples with
sum exactly $\alpha/d-n$.
On the other hand, 
\begin{equation}
M={\alpha+n \choose n}. \label{3.5}
\end{equation}
Then,
$$
	\frac{M\alpha}{\Delta}\le 
	\frac{(\a+1)\cdots(\a+n)d(n+1)}{(\a-d)\cdots(\a-nd)} 
	\le d(n+1)\left(\frac{\alpha+n}{\alpha-nd}\right)^{n}
$$
If we choose 
$$\alpha= d\lceil2(n+1)(nd+n)(2^{n}-1)\varepsilon^{-1}\rceil+3nd,$$ 
then $\alpha$ is divisible by $d$ and one finds
\begin{align*}
	\left(\frac{\alpha+n}{\alpha-nd}\right)^{n}&=\left(1+\frac{n+nd}{\alpha-nd}\right)^{n}
=1+ \sum_{r=1}^{n}{ n\choose r}\left(\frac{nd+n}{\a-nd}\right)^r\\
&\le 1+ (2^{n}-1)\frac{nd+n}{\a-nd} \le 1+\frac{\varepsilon}{2d(n+1)},
\end{align*}
and hence
$$
	\left(qd- \frac{M\alpha}\Delta \right) T_f(r) 
\ge d\left(q-n-1-\frac\varepsilon 2\right)T_f(r).
$$
Thus,
\begin{equation}\label{alpha}
	\left(qd- \frac{M\alpha}\Delta \right) T_f(r) 
	-o\big(T_f(r)\big)
\ge d(q-n-1-\varepsilon )T_f(r),
\end{equation}
for all $r$ sufficiently large.

We  now estimate $\sum_{j=1}^q N_f(r,Q_j) -\frac1\Delta  N_W(r,0)$
on the right hand side of (\ref{3.11}).
For each $z \in \C$, without loss of generality, we assume that $Q_j \circ f$ vanishes at $z$ for $1 \leqslant j \leqslant q_1$ and $Q_j \circ f$ does not vanish at $z$ for $j > q_1$. By the hypothesis that the $Q_j$ are 
in general position, we know $q_1 \leqslant n$. There are integers $k_j \geqslant 0 $ and nowhere vanishing holomorphic functions $g_j$ in a neighborhood $U$ of $z$ such that 
$$Q_j \circ f = (\zeta-z)^{k_j}g_j, \text{ for } j =1, \dots ,q,$$
where $k_j = 0$ if $q_1 < j \leqslant q.$ For $\{Q_1, \dots ,Q_n\} \subset \{Q_1, \dots ,Q_q\}$, we can obtain a basis $\{\psi_1, \dots ,\psi_M\}$ of $V_\alpha$ and linearly independent linear forms $L_1, \dots ,L_M$ such that $\psi_t(f) = L_t(F)$.
From basic properties of Wronskians,
\begin{align*}
W &= W(F_1, \dots ,F_M) = CW(L_1(F), \dots ,L_M(F)) \\
& = C 
\begin{vmatrix}\psi_1(f) & \dots &\psi_M(f) \\ (\psi_1(f))' & \dots &(\psi_M(f))' \\
\vdots &\ddots &\vdots \\
(\psi_1(f))^{(M-1)} & \dots &(\psi_M(f))^{(M-1)}
 \end{vmatrix}.
\end{align*}
Let $\psi$ be an element of basis, constructed with respect to $W_{({\bf i})}/W_{({\bf i'})}$, so we may write $\psi = Q_1^{i_1}\cdots Q_n^{i_n}\eta,\,  \eta \in V_{N-d\sigma({\bf i})}$. We have
$$\psi(f) = (Q_1(f))^{i_1}\cdots (Q_n(f))^{i_n}\eta(f),$$
where $(Q_j(f))^{i_j} = (\zeta-z)^{i_j.k_j}g_j^{i_j}, j =1, \dots ,n.$
Also we can assume that $k_j \geqslant M$ if $1 \leqslant j \leqslant q_0$ and  $1 \leqslant k_j < M$ if $q_0 < j \leqslant q_1$. And we observe that there are $\Delta_{({\ii})}$ such $\psi$ in our basis. Thus $W$ vanishes at $z$ with order at least
\begin{align*}
\sum_{({\ii})} \bigg(\sum_{j=1}^{q_0} i_j(k_j - M) \bigg) \Delta_{({\ii})} = \sum_{({\ii})}i_j \Delta_{({\ii})} \sum_{j=1}^{q_0}(k_j - M)=\Delta \sum_{j=1}^{q_0}(k_j - M).
\end{align*}
Therefore,
\begin{align}\sum_{j=1}^q N_f(r,Q_j) -\frac1\Delta  N_W(r,0)\le \sum_{j=1}^q N_f^M(r,Q_j).\label{b}\end{align}

Combining  the formulas (\ref{3.11}),   (\ref{b}) and (\ref{alpha}) together,
$$
d(q-n-1-\varepsilon)T_f(r)\le \sum_{j=1}^q N_f^M(r,Q_j),
$$
for all $r$ sufficiently large outside of a set of finite Lebesgue measure.

The proof is completed by (\ref{3.5}) for $0< \varepsilon< 1$:
\begin{align*}
M &\le \frac{(\alpha+n)^n}{n!}
\le \frac{\left(d\lceil2(n+1)(nd+n)(2^{n}-1)\varepsilon^{-1}\rceil+4nd)\right)^n}{n!}\\
&\le 2d\lceil2^n(n+1)n(d+1)\varepsilon^{-1}\rceil^n.
\end{align*}

\end{document}